\documentclass[reqno,12pt,a4paper]{amsproc}
\usepackage{amsthm,wrapfig,amsmath,amssymb,amsfonts,setspace,verbatim,graphicx,mathtools,mathrsfs,commath,float,mdframed,frame,xcolor,qtree,enumitem, ulem,xfrac}
\usepackage[hidelinks]{hyperref}
\usepackage[T1]{fontenc}

\usepackage[margin=1in]{geometry}

\DeclareMathOperator{\bur}{bur}
\DeclareMathOperator{\diam}{diam}
\DeclareMathOperator{\ind}{ind}
\DeclareMathOperator{\inter}{int}

\newtheorem{ut}{Theorem}

\newtheorem{ul}[ut]{Lemma}

\newtheorem{ucl}[ut]{Claim}

\newtheorem*{utmA}{Theorem A}
\newtheorem*{utmB}{Corollary B}
\newtheorem*{utmC}{Corollary C}

\theoremstyle{definition}
\newtheorem{ur}{Remark}

\newtheorem{ue}{Example}

\author[Lipham,  van Mill, Tuncali,  Tymchatyn,  Valkenburg]{David Lipham, Jan van Mill, Murat Tuncali, Ed Tymchatyn, Kirsten Valkenburg}

\address{Department of Mathematics and Data Science, College of Coastal Georgia, Brunswick GA 31520, United States of America}
\email{dlipham@ccga.edu}

\address{KdV Institute for Mathematics, University of Amsterdam, Science Park 105-107, 1090 GE Amsterdam, The Netherlands}
\email{j.vanmill@uva.nl}

\address{Department of Computer Science and Mathematics, Nipissing University, 100 College Drive, North Bay, ON, P1B 8L7, Canada}
\email{muratt@nipissingu.ca}

\address{Department of Mathematics and Statistics, University of Saskatchewan, 106 Wiggins Road, Saskatoon, SK, S7N 5E6, Canada}
\email{tymchat@math.usask.ca}
\email{kirstenvalkenburg@gmail.com}

\title{Buried points of plane continua}



\begin{document}

\begin{abstract}
Sets on the boundary of a complementary component of a continuum in the plane have been of interest since the early 1920’s. Curry and Mayer defined the buried points of a plane continuum to be the points in the continuum which were not on the boundary of any complementary component. Motivated by their investigations of Julia sets, they asked what happens if the set of buried points of a plane continuum is totally disconnected and non-empty. Curry, Mayer and Tymchatyn showed  that in that case the continuum is Suslinian, i.e.\ it does not contain an uncountable collection of non-degenerate pairwise disjoint subcontinua. In an answer to a question of Curry et al, van Mill and Tuncali constructed a plane continuum whose buried point set was totally disconnected, non-empty and one-dimensional at each point of a countably infinite set. In this paper we show that the van Mill-Tuncali example was the best possible in the sense that whenever the buried set is totally disconnected,  it is one-dimensional at each of at most countably many points. As a corollary we find that the buried set  cannot be almost zero-dimensional unless it is zero-dimensional. We also construct locally connected van Mill-Tuncali type examples. 
\end{abstract}

\maketitle

\section{Introduction}

A continuum is a compact, connected metric space. A point $x$ in a plane continuum $X\subset \mathbb R^2$ is \textbf{buried} if $x$ is not in the boundary (or frontier) of any component of $\mathbb R ^2\setminus X$.  We denote by $\bur(X)$ the set of all buried points of $X$. Note that $\bur(X)$ is a $G_{\delta}$-set.

Motivation to study buried points  of a continuum  comes from complex dynamics. Specifically: One of the  difficult problems in complex dynamics asks whether $\bur(J(R))$ must be zero-dimensional in case it is  punctiform, where $J(R)$ is  the Julia set of  a rational function on the sphere $\mathbb C_\infty$.  The Devaney-Rocha examples of Sierpiński gasket like Julia sets have this property.  Curry, Mayer and Tymchatyn proposed to study it from a purely topological point of view. For more details, see \cite{cr,fee,dev}.

In \cite[Proposition 3.1]{w4} van Mill and Tuncali prove that the set of buried points of a plane continuum can be a Cantor set. They use this to provide in §4 of their paper an example of a plane Suslinian continuum whose set of buried points is totally disconnected and weakly one-dimensional, but not zero-dimensional. Their example answers in the negative Question 1 of Curry, Mayer and Tymchatyn \cite{fee}, concerning whether or not a set of buried points is zero-dimensional provided it is totally disconnected. Note that if a plane continuum is regular (that is, has a basis of open sets with finite boundaries), then the set of buried points is either empty or zero-dimensional (see \cite[Theorem 3]{fee}).
  
Curry et al. showed that if $X$ is a plane continuum with complementary components whose boundaries are locally connected and $\bur(X)$ is Suslinian, then in fact $X$ has to be Suslinian \cite[Theorem 6]{fee}. This is the case for the example in \cite[§4]{w4} (the buried point set in that example is even totally disconnected). However, it  is not locally connected.  A plane continuum $X$ is locally connected if and only if the boundaries of complementary components are locally connected and form a null-sequence \cite[VI Theorem 4.4]{why}. Using this fact, we construct   in §5 a locally connected plane continuum whose set of buried points is totally disconnected but not zero-dimensional. The set at which the buried points are $1$-dimensional is countable. This is sharp according to the following theorem, which we prove in §3:
  
\begin{utmA}Let $X$ be a Suslinian plane continuum.  If  $Y\subset X$ is a totally disconnected Borel set, then $Y$ is zero-dimensional at all but countably many points.\end{utmA}

\begin{utmB}Let $X$ a plane continuum such that the boundary of each component of $\mathbb R^2\setminus X$ is locally connected (e.g.\ $X$ is locally connected).  If $\bur(X)$ is totally disconnected, then $\bur(X)$ is zero-dimensional at all but countably many points.\end{utmB}

Countable sets are zero-dimensional, so under the assumptions of Corollary B the buried set is either zero-dimensional or weakly $1$-dimensional. There is no almost zero-dimensional, weakly $1$-dimensional space  \cite[Theorem 1]{co2h}. Therefore $\bur(X)$ cannot be almost zero-dimensional in the proper sense (cf.\ \cite[Question 2.7]{cr}). To summarize:

\begin{utmC}Let $X$ be a plane continuum such that each component of $\mathbb R^2\setminus X$ has locally connected boundary (e.g.\ $X$ is locally connected).
\begin{itemize}
\item[\textnormal{(i)}] If $\bur(X)$ is totally disconnected, then $\bur(X)$ is at most weakly $1$-dimensional.
\item[\textnormal{(ii)}]  If $\bur(X)$ is almost zero-dimensional, then $\bur(X)$ is zero-dimensional. 
\end{itemize}\end{utmC}

 It is still unknown whether the Julia set of a rational function may have a buried set which is totally disconnected but not zero-dimensional (cf.\ \cite[Question 3]{fee}). We expect that Theorem A will have  other applications, for example, relating to endpoints of plane dendroids and chainable continua.

\section{Definitions}

A space $X$ is \textbf{Suslinian} if each collection of pairwise disjoint non-degenerate continua in $X$ is  countable. 

A space $X$  is \textbf{totally disconnected} if every two points of $X$ are contained in disjoint clopen sets, and \textbf{zero-dimensional} if $X$ has a basis of clopen sets.  Respectively, $X$ is  \textbf{zero-dimensional at} $x\in X$ if the point $x$ has a neighborhood basis of clopen sets (written $\ind_x X=0$ in   \cite[Problem 1.1.B]{engd}). 

A space $X$ is \textbf{almost zero-dimensional} if $X$ has a basis of open sets whose closures are intersections of clopen sets \cite{ov2,co2h}. Observe that every almost zero-dimensional space is totally disconnected.

The \textbf{dimensional kernel} $\Lambda(X)$ of a $1$-dimensional space $X$ is defined to be the set of all points at which $X$ is not zero-dimensional. Equivalently,  $$\Lambda(X)=\{x\in X:\ind_x X=1\}.$$ A $1$-dimensional space $X$ is  \textbf{weakly $\mathbf{1}$-dimensional} if $\Lambda(X)$ is zero-dimensional.

\section{Main Result}

Suppose that $X$ is a continuum and $Y\subset X$.   For each $y\in Y$ let $\mathcal U_y$ be the collection of all  open subsets $U$ of $X$ such that $y\in U$ and $\partial U\subset X\setminus Y$.  For each $y\in Y$ put 
$$F(y) =\bigcap_{U\in \mathcal U_y} \overline{U}.$$

\begin{ul}If $Y\subset X$ is totally disconnected and $y\in\Lambda(Y)$, then 
$F(y)$ is a non-degenerate subcontinuum of $X$ and $F(y)\cap Y = \{y\}$.\end{ul}

\begin{proof}Suppose $F(y)=\{y\}$ is degenerate. We show that $Y$ is zero-dimensional at $y$. Let $V$ be an arbitrary relatively open subset of $Y$ that contains $y$. Pick
an open subset $W$ of $X$ such that $W \cap Y = V$. Since $F(y) = \{y\} \subset V$,
there is by compactness an element $U \in \mathcal U_y$ such that $\overline{U} \subset W$.
But then $y \in U \cap Y \subset V$, and $U \cap Y$ is clopen in $Y$.

Now suppose that $y\in Y$ and  $z\in Y\setminus \{y\}$. Pick by total disconnectedness of $Y$ clopen subsets $C_0$ and $C_1 = Y \setminus C_0$
of $Y$ such that $y \in C_0$ and $z \in C_1$. There are disjoint open subsets $V_0$
and $V_1$ of $X$ such that $V_0 \cap Y = C_0$ and $V_1 \cap Y = C_1$. Observe that the
boundary of $V_0$ is contained in $X \setminus Y$. Since $V_1 \cap V_0 = \varnothing$, this shows
that $z \notin F(y).$ Hence $F(y)\cap Y = \{y\}$.

It remains to prove $F(y)$ is connected. Just suppose $F(y) = A \cup B$
where $A$ and $B$ are disjoint closed sets with $y \in A$. Let $U$ and $V$ be disjoint
open neighborhoods of $A$ and $B$ in $X$, respectively. By compactness there exists $W \in \mathcal U_y$ such that 
$W \subset U \cup V$. Then $W \cap U \in \mathcal U_y$, so $B = \varnothing.$\end{proof}

\begin{ur}If $X$ is locally connected and $U\in \mathcal U_y$, then the connected component of $y$ in $U$ also belongs to $\mathcal U_y$.  For this reason we can (and will) assume that each $\mathcal U_y$ consists entirely of path-connected open subsets of $X$. \end{ur}

\renewcommand{\i}{\mathrm{i}}

\begin{ul}Let $S$ be a circle in the plane, and $a_1,a_2,a_3,a_4$  four points of $S$ in cyclic order. Let $A_i$, $i\in \{1,2,3,4\}$,  be pairwise disjoint arcs intersecting $S$ only at an endpoint $a_i$, with opposite endpoints $b_i$ all in the bounded component of $\mathbb R ^2\setminus S$ denoted $\inter(S)$. If $K\subset \mathbb R ^2\setminus (A_2\cup S\cup A_4)$ is a continuum containing $b_1$ and $b_3$, then $b_2$ and $b_4$ are in different components of $\mathbb R^2\setminus (A_1\cup S\cup A_3\cup K)$. \end{ul}

\begin{proof}We may assume that $S$ is the unit circle in the complex plane $ \mathbb C$, and $$a_1=1,a_2=\i , a_3=-1,a_4=-\i .$$ Let $\gamma\subset \mathbb C\setminus (A_1\cup S\cup A_3)$ be any arc from $b_2$ to $b_4$. We will show that $\gamma$ intersects $K$. This will  imply that  $b_2$ and $b_4$ are in different components of $\mathbb C\setminus (A_1\cup S\cup A_3\cup K)$, as all such components are path-connected.

In $\gamma\cup A_2\cup A_4$ there is an arc $\overline\gamma$ intersecting $S$ only at endpoints $\pm \i$. Let $\alpha$ and $\beta$ be the left and right halves of the circle $S$. By the $\theta$-curve theorem \cite[VI  1.6]{why} (or  Munkres Lemma 64.1) $\text{int}(S)\setminus \overline\gamma$ is the union of two disjoint domains  $U$ and $V$ whose boundaries are  $\alpha\cup \overline\gamma$ and $\overline\gamma\cup \beta$, respectively. Let $W$ be an open ball centered at $-1$ that misses $\overline\gamma$. Then $W\cap \text{int}(S)\subset U\cup V$. Since $W\cap \text{int}(S)$ is connected and $-1\in \partial U$, we have $W\cap \text{int}(S)\subset U$.  Thus $U\cap A_3\neq\varnothing$. It follows that $A_3\setminus \{-1\}\subset U$. Thus $b_3\in U$. Likewise, $b_1\in V$. Now $K$ is a connected subset of $\text{int}(S)$ meeting both $U$ and $V$. So $K\cap \overline\gamma\neq\varnothing$.  Therefore $K\cap \gamma\neq\varnothing$.\end{proof}

We can now prove Theorem A for locally connected $X$.

\begin{ut}Let $X$ be a locally connected Suslinian plane continuum. If $Y\subset X$ is a totally disconnected Borel set, then $\Lambda(Y)$ is countable.\end{ut}

For a contradiction suppose that $Y\subset X$ is totally disconnected and  $Z=\Lambda(Y)$ is uncountable.   Then by Lemma 1 there exists $\varepsilon>0$ such that for uncountably many $z$'s, $\diam(F(z))\geq5\varepsilon$. The set of all $z$'s with this property is closed in $Z$, which is Borel, and thus it contains a Cantor set. Let us just assume that $Z $ is any Cantor set in $Y$  with the property $\diam(F(z))\geq 5 \varepsilon$  for each $z\in Z$. We may further assume that $Z$ lies inside of an open ball of radius $\varepsilon$, centered at some point of $Z$.  Let $S$ and $S'$ be the circles of radii $\varepsilon$ and $2\varepsilon$, respectively, centered at the point. Observe that every $F(z)$ crosses $S'$.
 
In the claims  below, all boundaries are respective to $X$ (not $\mathbb R ^2$).

\begin{ucl}Let $U$ be an open subset of $X$ intersecting $Z$ with $\partial U\subset (X\setminus Y)\cup S$. Then for every point $p\in \mathbb R^2$ there exists a connected open set $W\subset U$ intersecting $Z$ such that $p\notin \overline W$ and   $\partial W\subset (X\setminus Y)\cup S$.\end{ucl}

\begin{proof}For each $z$ in the uncountable set $Z\cap U\setminus \{p\}$ there is a non-degenerate continuum $K(z)\subset F(z)$ containing $z$ and missing $S\cup \{p\}$. Just take any closed neighborhood $N$ of $z$ missing $S\cup \{p\}$, and let $K(z)$ be the component of $z$ in $F(z)\cap N$; by the boundary bumping theorem \cite[Theorem 5.4]{nad} $K(z)$ is non-degenerate. Note that $K(z)\cap Y=\{z\}$ by Lemma 1. By the Suslinian property of $X$  there exist $$z_1,z_2,z_3,z_4\in Z\cap U\setminus \{p\}$$ such that $K(z_i)\cap K(z_1)\neq\varnothing$ for each $i\in \{2,3,4\}$. Let $U_i\in \mathcal U_{z_i}$ be pairwise disjoint. Since $F(z_i)$ extends beyond $S'$, so does $U_i$. Thus there is an arc $A_i\subset U_i$ from $z_i$ to $a_i\in S'$ that intersects $S'$ only at $a_i$. By a permutation of indices $i\in \{2,3,4\}$, we may assume 
  that the $a_i$'s are in cyclic order around $S'$. Let $K=K(z_1)\cup K(z_3)$. Note that $p\notin K$.%(If not in cyclic order, replace $z_3$ with $z_i$ such that $a_i$ is cyclically 2 away from $a_1$).

\uline{Case 1}: If $p\in U_1\cup S'\cup U_3$,  simply let $W$ be the component of $z_2$ in $U_2\cap U\setminus S$. Then $W$ is a connected open subset of $U$ intersecting $Z$, with $p\notin \overline W$. Note also that $W$ is clopen in $U_2\cap U\setminus S$, which means that  $\partial W\subset \partial(U_2\cap U\setminus S)$.  The boundary of any finite intersection of open sets is contained in the union of the individual boundaries. Therefore $$\partial W\subset \partial U_2 \cup \partial U\cup \partial (X\setminus S)\subset (X\setminus Y)\cup S.$$

\uline{Case 2}: Else $p\in O=\mathbb R^2\setminus (A_1\cup S'\cup A_3\cup K)$. 

By Lemma 2, $z_2$ and $z_4$ are in different components of $O$.   Of these two components, at least one does not contain $p$. Let's say that $V$ is the component of $z_2$ in $O$, and $p\notin V$. The  component of $p$ in $O$ is an open set missing $V$, so $p\notin \overline V$. Let  $W$ be the component of $z_2$ in $U_2\cap U \cap V\setminus S$.  Then $W$ is a connected open subset of $U$ intersecting $Z$, and $p\notin \overline {W}$.   It remains to show that $\partial W\subset (X\setminus Y)\cup S$. Note the following:
 \begin{itemize}\renewcommand{\labelitemi}{\Large$\cdot$}
 \item $\partial V\subset A_1\cup S'\cup A_3\cup K$;
 \item $\overline W$ misses $A_1\cup A_3$ because $W\subset U_2$;
 \item $\overline W$ misses $S'$ because it lies in the closed disc bounded by $S$.
 \end{itemize}
 Hence $\overline{W}\cap \partial V\subset K\setminus \{z_1,z_3\}\subset X\setminus Y.$  We conclude that \begin{align*}
\partial W&\subset \overline W\cap [\partial U_2\cup \partial U \cup \partial V\cup \partial (X\setminus S)] \\
&\subset \partial U_2\cup \partial U \cup (\overline W\cap \partial V)\cup \partial (X\setminus S)\\
&\subset (X\setminus Y)\cup S.\hfill\qedhere\end{align*}\end{proof}

\vspace{-.2cm}
\begin{figure}[h]\includegraphics[scale=0.32]{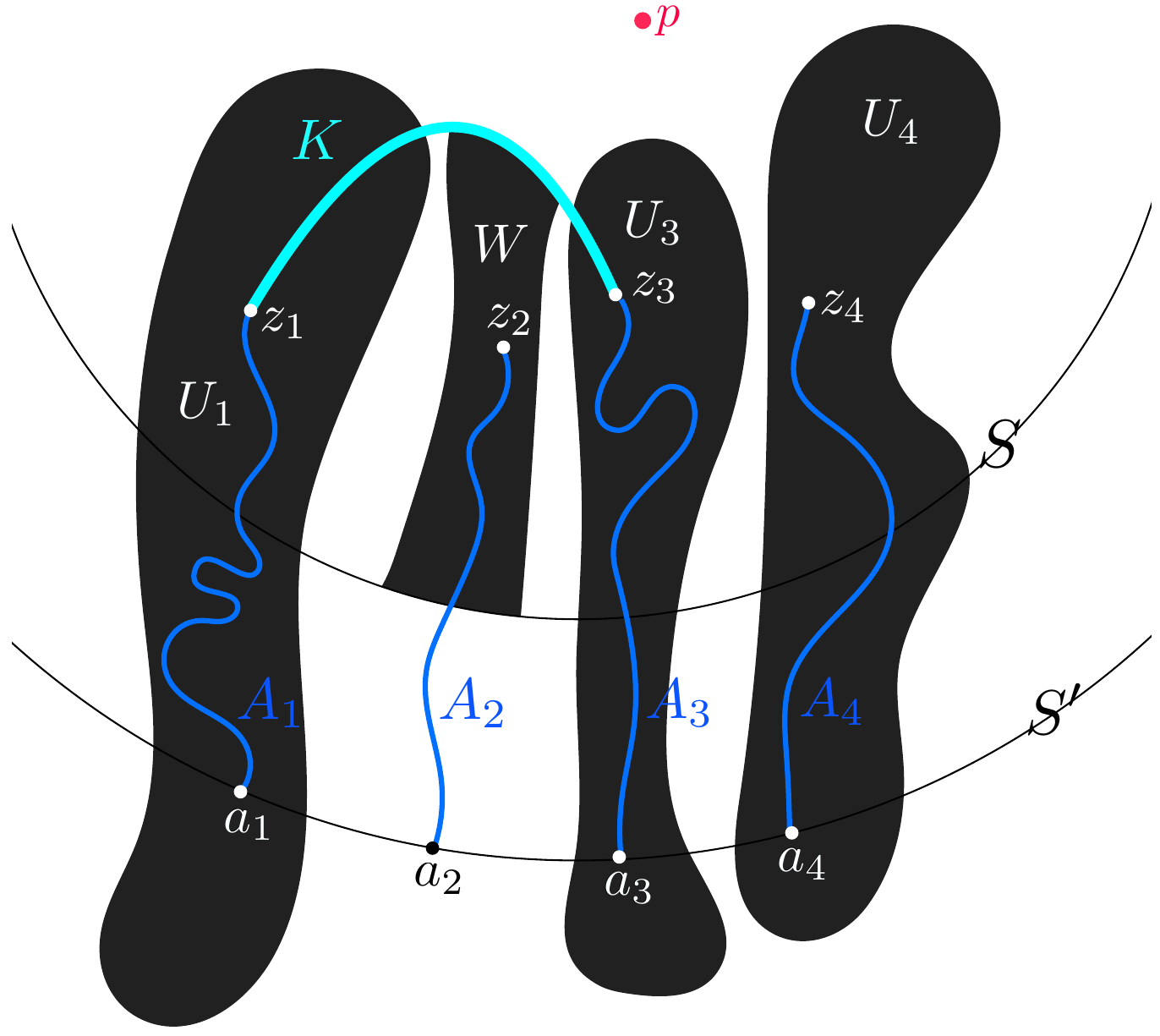}\hspace{.5cm}
\includegraphics[scale=0.32]{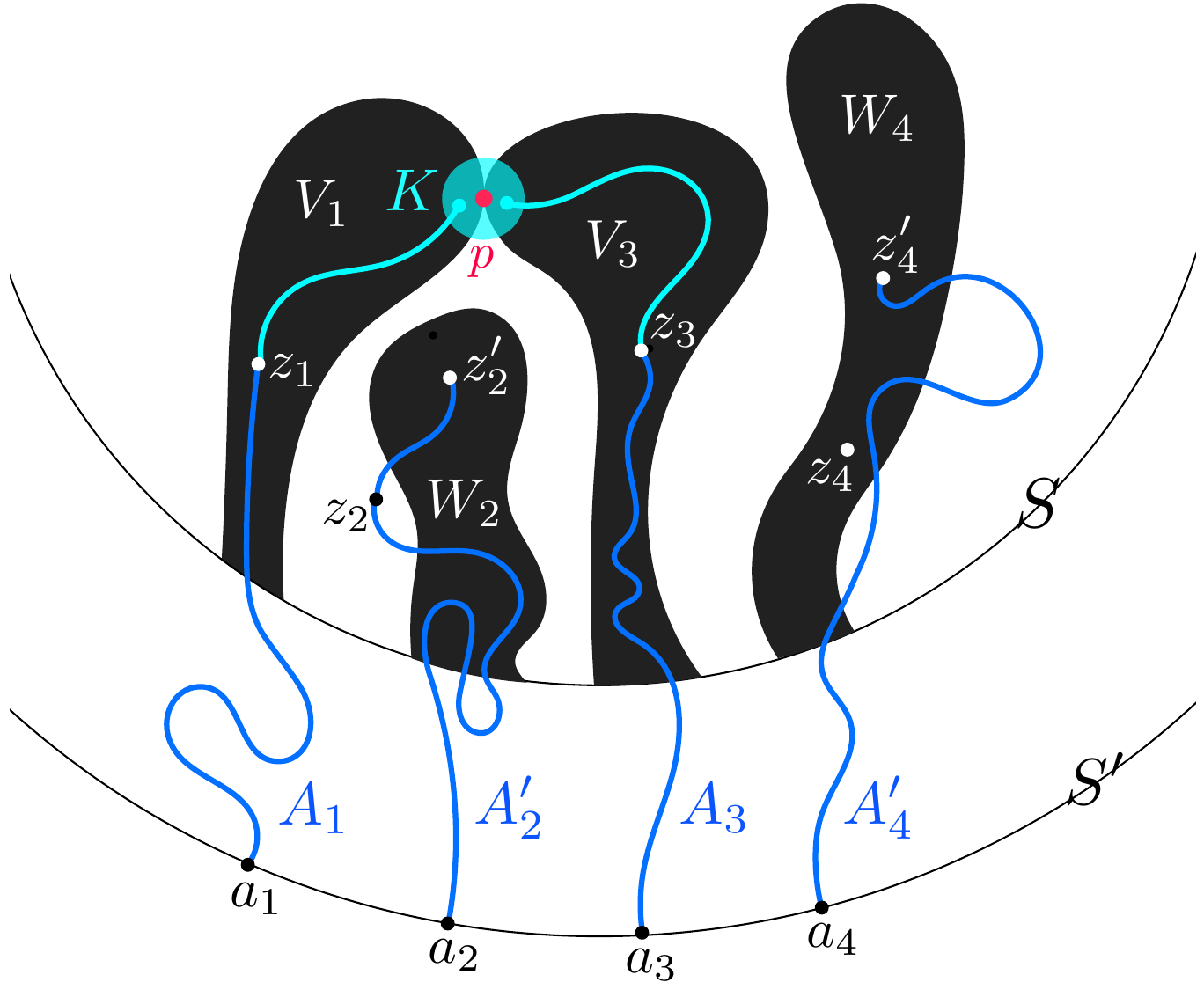}
\caption{Proofs of Claim 4 (\textit{left}) and Claim 5 (\textit{right})}
\end{figure}

\begin{ucl}Let $U$ be an open subset of $X$ intersecting $Z$ with  $\partial U\subset (X\setminus Y)\cup S$. Then there are connected open subsets $W_1$ and $W_2$ of $U$ each intersecting $Z$, with disjoint closures, such that  $\partial W_i\subset (X\setminus Y)\cup S$. \end{ucl}

\begin{proof}Let $z_1,z_2,z_3,z_4\in Z\cap U$. Let $U_i\in \mathcal U_{z_i}$ be pairwise disjoint. Let $A_i$ be an arc in $U_i$ from $z_i$ to  $a_i\in S'$,  intersecting $S'$ only at $a_i$. Assume that the $a_i$'s are in cyclic order.  Let $V_i$ be component of $z_i$ in $U_i\cap U\setminus S$. Assume that $\overline {V_1}$ and $ \overline {V_3}$ have a common point $p$ (if they are disjoint then we're done). Note that $p\notin U_2\cup U_4$, so $p\notin A_2\cup A_4\cup V_2\cup V_4$.

By Claim 4, there exist $z_2',z_4'\in Z$ and connected open subsets $W_2$ and $W_4$  of $V_2$ and $V_4$ whose closures miss $p$, such that $z_i'\in W_i$ and $\partial W_i\subset (X\setminus Y)\cup S$. For each $i\in \{2,4\}$,  working within $V_i$ we can modify $A_i$ to an arc $A_i'$ that ends at $z_i'$  instead of $z_i$. The other endpoint of $A_i'$ is still $a_i$, and $p\notin A_i'$. Let $V$ be a  connected neighborhood of $p$ with closure missing $\overline{W_2}\cup \overline{W_4}\cup S'\cup A_2'\cup A_4'$. For each $i\in \{1,3\}$ let   $B_i$ be an arc in $V_i$ from $z_i$ into $V$.  Let $K=B_1\cup B_3\cup \overline V$.   By  Lemma 2, $z_2'$ and $z_4'$ are in different components of $\mathbb R^2\setminus(A_1\cup S'\cup A_3\cup K)$. The continua $\overline{W_2}$ and $\overline{W_4}$ are contained  in these components, hence they are disjoint. \end{proof}

By Claim 5 there are two connected open sets $W_{\langle 0\rangle}$ and $W_{\langle 1\rangle}$ in $X\setminus S$, intersecting $Z$, with disjoint closures and boundaries in $(X\setminus Y)\cup S$. Assuming that $\alpha\in 2^{<\mathbb N}$ is a finite binary sequence and $W_{\alpha}$ has been defined, apply Claim 5 to get connected open  subsets $W_{\alpha^\frown 0}$ and $W_{\alpha^\frown 1}$ of $W_\alpha$, each intersecting $Z$, with disjoint closures and boundaries in $(X\setminus Y)\cup S$. Their boundaries must meet $S$ because each $F(z)$ meets $S'$. So for every infinite sequence $\alpha\in 2^{\mathbb N}$, $$K_{\alpha}=\bigcap_{n=1}^\infty \overline{W_{\alpha\restriction n}}$$ is a continuum in $X$  stretching  from $Z$ to $S$. Clearly if $\alpha\neq\beta$ then $K_{\alpha}$ and $K_{\beta}$ are disjoint. Therefore $\{K_\alpha:\alpha\in 2^{\mathbb N}\}$ is an uncountable collection of pairwise disjoint non-degenerate subcontinua of $X$, a contradiction to  the Suslinian property of $X$. Hence $\Lambda(Y)$ must have been countable. This completes the proof of Theorem 3. \hfill$\blacksquare$

\medskip

Theorem A is a direct consequence of Theorem 3 and:

\begin{ut}Every Suslinian plane continuum is contained in a locally connected Suslinian plane continuum. \end{ut}

\begin{proof}Let $X$ be a Suslinian plane continuum, and $U_0,U_1,U_2,\ldots$ the components of $\mathbb R^2\setminus X$ with $U_0$ unbounded. For each $m\geq 0$,  in the domain $U_m$ there is a sequence $S^m_0,S^m_1\ldots$ of  disjoint, concentric simple closed curves limiting to the boundary of $U_m$. Put $D^0_0=\varnothing$ and for $m\geq 1$ let $D^m_0$ be the closed topological disc bounded by $S^m_0$. For each $m\geq 0$ and  $n\geq 1$ let $D^m_n$ be the closed annulus bounded by $S^m_{n-1}$ and $S^m_{n}$. In $D^m_n$ there is a finite collection of arcs $\mathcal I^m_n$ covering $\partial D^m_n$ such that each component of $D^m_n\setminus \bigcup  \mathcal I^m_n$ has diameter less than $\sfrac{1}{m+n}$.  We can easily arrange that the boundaries of these components are simple closed curves, and $X^m_n=\bigcup  \mathcal I^m_n$ is connected.  Let $$X'=X\cup \bigcup_{m=0}^\infty\bigcup_{n=0}^\infty X^m_n.$$ The components of $\mathbb R^2\setminus X'$ form a null-sequence, and their boundaries are simple closed curves. Hence by \cite[VI Theorem 4.4]{why},  the continuum $X'$ is locally connected. Because $X' \setminus X$ is a countable union of arcs, $X'$ is Suslinian.
\end{proof}

\section{Corollaries}

As indicated in §1, Corollary B follows from Theorem A and \cite[Theorem 3.5]{fee}, while Corollary C follows from Corollary B and \cite[Theorem 1]{co2h}.

\section{Examples}

First is a locally connected extension of the continuum  in \cite{w4} with the same buried set:  the totally disconnected and (weakly) $1$-dimensional space $K$  by Kuratowski. A  brief description of $K$ is given in the remark following the example.

\begin{ue}There exists a locally connected plane continuum whose buried set is totally disconnected and $1$-dimensional.\end{ue}

\begin{proof}The continuum $Z$ in \cite{w4} 
 is constructed such that $\mathbb R^2 \setminus Z = \bigcup \mathcal W$, where $\mathcal W$ consists of a countable number of disjoint connected open subsets of $\mathbb R^2$ with simple closed  curves as boundaries. One component, say $W_0$, is unbounded and the others, say  $W_1, W_2, \ldots$, form a sequence  of bounded domains. Let $\mathcal I_n$, for $n \geq1$, be a finite set of arcs in $W_n$  such that the collection  $\mathcal V_n$ of components of $W_n \setminus \bigcup \mathcal I_n$  has mesh less than $\sfrac{1}{n}$. We can easily arrange that the boundary of every element of $\mathcal V_n$ is a simple closed curve, and $Z_n=\partial W_n\cup \bigcup \mathcal I_n$ is connected. As in Theorem 6, 
$$Z' =Z\cup \bigcup_{n=1}^\infty Z_n$$ is a locally connected continuum. Clearly, $\bur(Z) \subset \bur(Z')$. Points in $Z' \setminus Z$ have finite graph neighborhoods in $Z'$, so are not in $\bur(Z')$. If $z \in Z \setminus \bur(Z)$, then $z\in \partial W_n$ for some $n$. Let $V_1, V_2,\ldots, V_k$ be components of $W_n \setminus X_n$. Then $z \in \partial V_i$ for some $i \in \{1,\ldots, k\}$. So
$\bur(Z') = \bur(Z)=K$.
\end{proof}
\begin{figure}[h]
\includegraphics[scale=0.35]{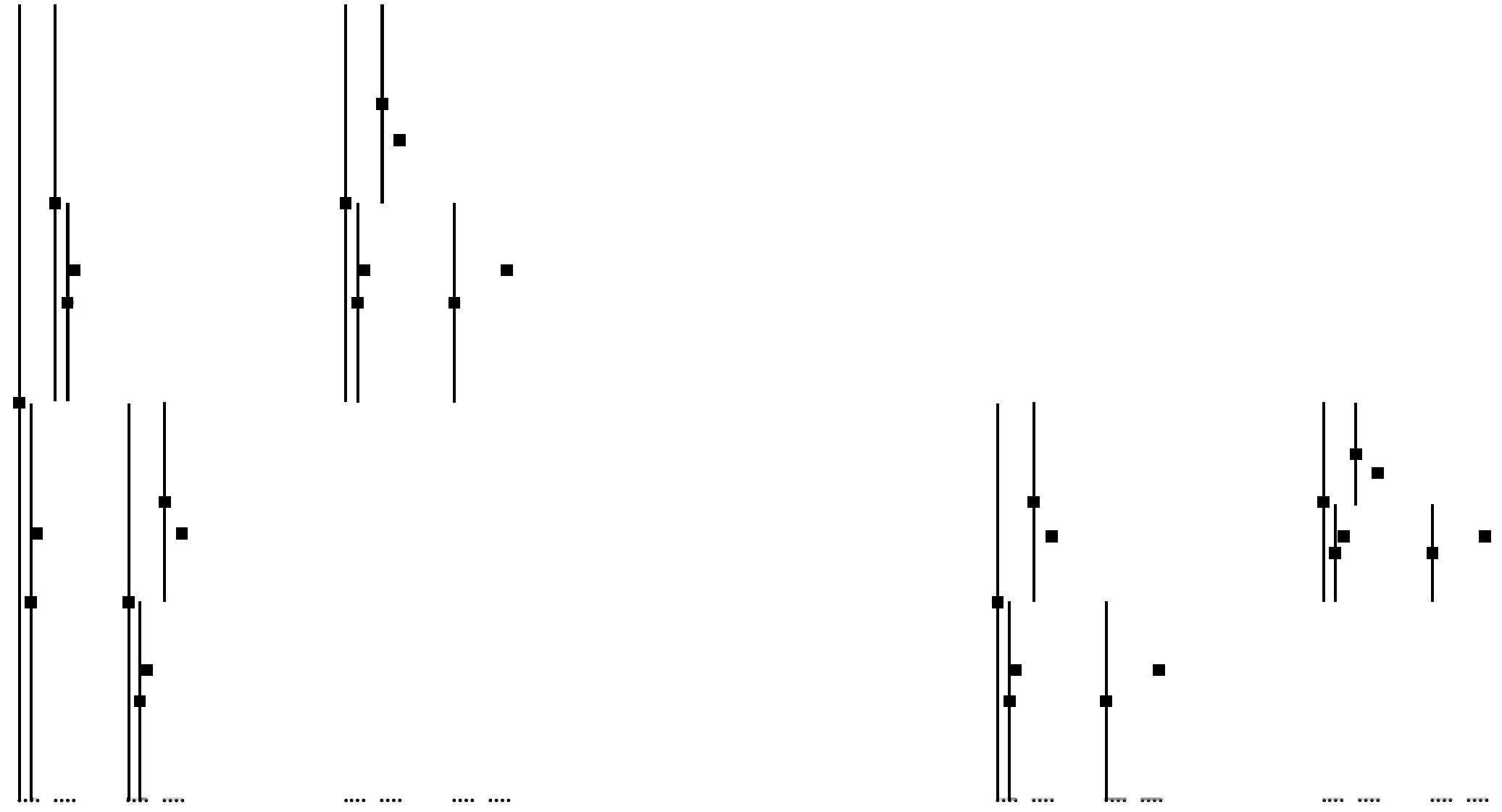}
\caption{The closure of $K$ consists of the points and vertical arcs. The vertical arcs are supported by $C_1$, while the isolated points are over  $C_0$. }
\end{figure}
\begin{ur}Kuratowski's space $K$ is the graph of a particular function $f:C\to [-1,1]$ defined on the  middle-thirds Cantor set $C$. The function $f$  has a simple algebraic formula based on   binary representations of members of $C$ (see \S 4 of \cite{w4}).

If $C_1$ is the countable set consisting of $0$ and all right endpoints of the intervals $$\textstyle(\frac{1}{3}, \frac{2}{3}),(\frac{1}{9}, \frac{2}{9}),(\frac{7}{9}, \frac{8}{9}),\ldots$$ removed from $[0,1]$ to obtain $C$, then for every $x\in C_1$ the space $K$  is $1$-dimensional at $\langle x,f(x)\rangle$. On the other hand, $f$ is continuous at each point of   $C_0=C\setminus C_1$. Thus if $x\in C_0$ then $K$ is zero-dimensional at $\langle x,f(x)\rangle$. See Figure 2.

The crux of \cite{w4} is that $\overline K$ lies inside of a plane continuum $Z$ with $\bur(Z)=K$. Below is an illustration of $K$ and its closure in $C\times [-1,1]$. \end{ur}

Lastly, we observe that Theorem A is false outside of the plane.

\begin{ue}There exists a Suslinian  dendroid whose endpoint set is totally separated and $1$-dimensional at every point.  \end{ue}

Such an example may be constructed from an upper semi-continuous decomposition of the Lelek fan.


\begin{thebibliography}{HD}

\bibitem{cr}C. P. Curry and J. C. Mayer, Buried points in Julia sets, J. Difference Equations Appl. 16 (2010), 435--441.

\bibitem{fee}C. P. Curry, J. C. Mayer and E. D. Tymchatyn, Topology and measure of buried points in Julia sets, Fundamenta Mathematicae 222 (2013), 1--17.

\bibitem{engd}R. Engelking, Dimension Theory, Volume 19 of Mathematical Studies, North-Holland, Amsterdam, 1978.

\bibitem{dev} R. L. Devaney, M. M. Rocha, and S. Siegmund, Rational maps with generalized Sierpiński gasket Julia sets, Topol. Appl. 154(1) (2007), pp. 11--27.

\bibitem{co2h}D. S. Lipham, A dichotomy for spaces near dimension zero, Topology Appl., Volume 336 (2023), 108611. 

\bibitem{w4} J. van Mill, M. Tuncali, Plane continua and totally disconnected sets of buried points. Proc. Amer. Math. Soc. 140 (2012), no. 1, 351--356.

\bibitem{nad}S. B. Nadler Jr., Continuum Theory: An Introduction, Pure Appl. Math., vol. 158, Marcel Dekker, Inc., New York, 1992.

\bibitem{ov2}L. G. Oversteegen, E. D.  Tymchatyn. On the dimension of certain totally disconnected spaces, Proc. Amer. Math. Soc. 122 (1994), 885--891.

\bibitem{why}G. T. Whyburn, Analytic topology, AMS Colloquium Publications, Volume 28. American Mathematical Society, New York, 1942.

\end{thebibliography}
\end{document}